\documentclass[a4paper,leqno, 11pt]{article}
\usepackage[utf8]{inputenc}
\usepackage{url}
\usepackage{amsmath}
\usepackage{amsthm}
\usepackage{amssymb}
\usepackage[english]{babel}
\usepackage{geometry}
\usepackage{tikz}
\usepackage{lipsum}
\usepackage{stackrel}
\usepackage{comment}
\usepackage{xcolor, graphicx}
\usepackage{float}

\definecolor{darkblue}{rgb}{0.05, .05, .65}
\definecolor{darkgreen}{rgb}{0.05, .70, .05}
\definecolor{darkred}{rgb}{0.8,0,0}
\definecolor{darkbrown}{rgb}{0.4, 0.26, 0.13}

\usepackage[colorlinks=true,linkcolor=darkblue,  % color de enlaces internos (secciones, ecuaciones)
urlcolor=black,   % color de enlaces URL
citecolor=darkgreen]{hyperref}

\usepackage{fancyhdr}

\usepackage{indentfirst}
\usepackage{enumitem}
\setlist[itemize]{noitemsep, topsep=0pt}

\usepackage[fixlanguage]{babelbib}
\newtheorem{theorem}{Theorem}[section]

\newtheorem{proposition}[theorem]{Proposition}
\newtheorem{corollary}[theorem]{Corollary}

\theoremstyle{remark}

\newtheorem*{remark2}{Remark}

\theoremstyle{definition}

\makeatletter
\renewcommand{\maketitle}{%
	\begin{center}
		{\Large\bfseries \@title \par}
		\vskip 1em
		{\normalsize \@author \par}
		\vskip 0.5em
		{\footnotesize \@date \par}
	\end{center}
	\vskip 1em
}

\newcommand{\authorinfo}[3]{%
	\par\addvspace{8pt}%
	\noindent\footnotesize
	\noindent\textsc{#1}\par
	\noindent\textit{#2}\par
	\noindent\texttt{#3}\par
}
\makeatother
\newcommand{\subjclass}[1]{%
	\par\addvspace{6pt}%
	\noindent\footnotesize\textbf{Mathematics Subject Classification (2020):}
	\ #1\par
}
\newcommand{\keywords}[1]{%
	\par\smallskip
	\noindent\footnotesize\textbf{Key words and phrases:}
	\ #1\par
}
\newcommand{\R}{\mathbb{R}}
\newcommand{\RN}{\mathbb{R}^{N}}
\newcommand{\Q}[1]{Q^{(#1)}}
\newcommand{\hQ}[3]{Q^{(#1)}_{#2#3}}
\newcommand{\mP}[3]{P^{(#1)}_{#2#3}}
\newcommand{\p}[1]{P^{(#1)}}
\newcommand{\suma}[2]{\mathop{\sum}\limits^{#2}_{#1}}
\newcommand{\Qn}[1]{\mathcal{Q}^{(#1)}_{n}}
\newcommand{\Pn}[1]{\mathcal{P}^{(#1)}_{n}}
\newcommand{\dist}[1]{{\rm d}(#1) }
\newcommand{\rd}{{\rm d}}
\newcommand{\lp}{\left(}
\newcommand{\rp}{\right)}
\newcommand{\lb}{\left[}
\newcommand{\rb}{\right]}
\newcommand{\lsb}{\left\{}
\newcommand{\rsb}{\right\}}
\newcommand{\copyrightnotice}[1]{%
	\gdef\thecopyright{#1}%
}
\newcommand{\printcopyright}{%
	\begingroup
	\small \thecopyright
	\endgroup
}
\newcommand{\thecopyright}{}

\numberwithin{equation}{section}

\begin{document}
	
	%%%%%%%%%%%%%%
	\title{The Pohozaev identity for the\\ Spectral Fractional Laplacian}
	\author{Itahisa Barrios-Cubas, Matteo Bonforte, Mar\'{i}a del Mar Gonz\'{a}lez and Clara Torres-Latorre}
	\date{}
	\maketitle
	
	\begin{abstract}
		In this paper, we prove a Pohozaev identity for the Spectral Fractional Laplacian (SFL). This identity allows us to establish non-existence results for the semilinear Dirichlet problem $(-\Delta|_{\Omega})^su = f(u)$ in star-shaped domains. The first such identity for non-local operators was established by Ros-Oton and Serra in 2014 for the Restricted Fractional Laplacian (RFL). However, the SFL differs fundamentally from the RFL, and the integration by parts strategy of Ros-Oton and Serra cannot be applied. Instead, we develop a novel spectral approach that exploits the underlying quadratic structure. Our main result expresses the identity as a Schur product of the classical Pohozaev quadratic form and a transition matrix that depends on the eigenvalues of the Laplacian and the fractional exponent.
	\end{abstract}
	
	%%%%%%%%%%%%%%
	\section{Introduction}
	
	Pohozaev identities are a fundamental tool in the analysis of elliptic PDEs, with applications ranging from non-existence results to geometric problems. Notable consequences include the non-existence results by Pucci and Serrin \cite{Pucci-Serrin-variational} or Br\'ezis and Nirenberg \cite{BrezisPositiveSolu1983Nirenberg} for semilinear PDEs, as well as Schoen's application \cite{Schoen-existence} to the Yamabe problem in conformal geometry, which is the higher-dimensional analogue of Kazdan and Warner's curvature prescription problem for surfaces \cite{Kazdan-Warner}.
	
	These identities can be understood as conservation laws arising from symmetries of the operator such as the scaling invariance of the Laplacian. They also encode the geometry of the domain $\Omega$ and the normal derivative of the function on $\partial\Omega$.
	
	Non-local versions extend these results to fractional and integro-differential equations, which model anomalous diffusion, long-range interactions, and non-local phenomena in physics and probability, see \cite{RosOton2024IntDiffFernandezReal}. The first Pohozaev identity for non-local operators was established in the influential paper by Ros-Oton and Serra \cite{ros2014pohozaev}, where they used the invariances of the fractional Laplacian to obtain a new Pohozaev formula. Their breakthrough consisted of finding the appropriate boundary term for the Restricted Fractional Laplacian (RFL) on a domain $\Omega$.
	
	However, it is not possible---to the best of our knowledge---to mimic the proof by Ros-Oton and Serra in the case of the Spectral Fractional Laplacian (SFL) in a domain. The situation is fundamentally different: the SFL is not invariant under dilations or translations, has a kernel that degenerates at the boundary, and sharp boundary regularity is not yet known. These obstructions have prevented the development of a Pohozaev identity for the SFL for over a decade.
	
	To overcome these obstacles, we propose a novel spectral approach based on the underlying quadratic structure to obtain a new Pohozaev identity for the SFL. We apply this identity to establish non-existence of non-trivial solutions for semilinear Dirichlet problems in star-shaped domains.
	
	\color{darkred}
	\normalcolor
	
	\vspace*{\fill}
	\noindent\footnotesize
	\subjclass{35J61, 35R11}
	\keywords{Semilinear elliptic PDE; Non-local operators; Pohozaev identity; Spectral Fractional Laplacian.}
	
	\fancypagestyle{firstpage}{
		\fancyhf{}
		\renewcommand{\headrulewidth}{0pt}
		\renewcommand{\footrulewidth}{0pt}
		\fancyfoot[C]{\printcopyright}
	}
	
	\copyrightnotice{\sl\footnotesize\copyright~2025 by the authors. This paper may be reproduced, in its entirety, for non-commercial purposes.}
	\thispagestyle{firstpage}
	
	\normalsize
	
	\newpage
	
	%%%%%%%%%%%%%%%%%%%%%%%
	\subsection{Main results}\label{section:main}
	Let $\Omega\subset\RN$ be a bounded $C^{1,1}$ domain and $s\in(0,1)$. 
	Let $\{\lambda_{k}, \varphi_{k}\}_{k=1}^\infty$ denote the eigenelements 
	of the classical Dirichlet Laplacian, where $\lambda_k$ are the eigenvalues 
	$0 < \lambda_1 \leq \lambda_2 \leq \dots$ with $\lambda_{k} \to \infty$ 
	(repeated according to multiplicity), and $\varphi_{k}$ are the eigenfunctions 
	satisfying
	\begin{equation*}
		\begin{cases}
			-\Delta \varphi_{k}=\lambda_{k}\varphi_{k}, & \text{in }\Omega,
			\\
			\varphi_{k}=0, & \text{on }\partial\Omega,
		\end{cases}
	\end{equation*}
	normalized to $\langle \varphi_{j},\varphi_{k}\rangle=\delta_{jk}$. Any $u \in L^{2}(\Omega)$ has Fourier coefficients $\hat{u}_{k} = \langle u, \varphi_k \rangle$ and expansion $u = \sum_{k=1}^{\infty} \hat{u}_{k} \varphi_{k}$. The Spectral Fractional Laplacian (SFL) is then defined as
	$$(-\Delta|_{\Omega})^{s}u = \sum_{j=1}^{\infty}\lambda_{j}^{s}\hat{u}_{j}\varphi_{j}.$$
	
	We recall the classical Pohozaev identity for the (local) Dirichlet Laplacian: 
	let $u\in C^{2}(\overline{\Omega})$ and $u=0$ on $\partial\Omega$, then
	\begin{equation} \label{eq: abstract Pohozaev local}
		Q^{(1)}[u] := \left(\frac{2-N}{2}\right)\int_{\Omega} u(-\Delta) u\,dx 
		- \int_{\Omega}(x\cdot \nabla u)(-\Delta)u\,dx
		=\frac{1}{2} \int_{\partial \Omega}\left|\frac{\partial u}{\partial\nu}\right|^{2}(x \cdot \nu)\,dS,
	\end{equation}
	where $\nu$ is the outward normal vector and $dS$ is the $(N-1)$-dimensional 
	surface measure. It is clear from the above expression that $Q^{(1)}[u]\ge 0$ 
	when the domain $\Omega$ is star-shaped. In terms of Fourier coefficients, 
	$Q^{(1)}$ can be written as the quadratic form 
	$Q^{(1)}[u] = \sum_{j,k} Q^{(1)}_{jk} \hat{u}_j \hat{u}_k$ with
	$$Q^{(1)}_{jk} = \frac{1}{2} \int_{\partial \Omega} (\nabla \varphi_{j} \cdot \nabla \varphi_{k})(x \cdot \nu)\,dS.$$
	
	Our main result is a Pohozaev identity for the SFL whose positivity is inherited from the classical Pohozaev identity and thus depends only on the geometry of the domain.
	
	\begin{theorem}[\bf Abstract Pohozaev identity for the SFL] \label{Pohozaev SFL}
		Let $\Omega$ be a bounded $C^{1,1}$ domain, $s\in (0,1)$ and $u\in H^{\max \lsb 2s, 3/2 \rsb }_0(\Omega)$. Define
		\begin{equation} \label{Pohozaev SFL id}
			\Q s [u]:=\lp\frac{2s-N}{2}\rp \int_{\Omega} u(-\Delta|_{\Omega})^{s} u\,dx - \int_{\Omega}(x \cdot \nabla u)(-\Delta|_{\Omega})^{s} u \,dx.
		\end{equation}
		Then $\Q s$ is a positive semidefinite quadratic form whenever $\Q 1$ is. In particular,
		\begin{equation*}
			\Q 1[u]\ge 0 \qquad\mbox{implies}\qquad\Q s [u]\geq 0.
		\end{equation*}
		Moreover, $\Q s$ can be expressed as
		\begin{equation}\label{conclusion-structure}
			\Q s =\p s\circ \Q 1,
		\end{equation}
		where $\circ$ denotes the Schur (or Hadamard) product. Equivalently, in coordinates,
		\begin{equation*}
			\Q s [u]=\suma{j,k=1}{\infty} \hQ sjk \hat{u}_{j} \hat{u}_{k}, \quad \text{where} \quad \hQ sjk= \mP s j k \hQ 1jk.
		\end{equation*}
		The transition matrix $\mP s{}{}$ is given by
		\begin{equation}\label{formula-P}
			\mP s j k=
			\begin{cases}
				\dfrac{\lambda_{j}^{s}-\lambda_{k}^{s}}{\lambda_{j}-\lambda_{k}}, & \mbox{if } \lambda_{j}\neq \lambda_{k}, \\[4mm]
				s\lambda_{j}^{s-1}, & \mbox{if } \lambda_{j} = \lambda_{k}.
			\end{cases}
		\end{equation}
	\end{theorem}
	
	\noindent\textsl{Application to Semilinear PDEs. }Theorem \ref{Pohozaev SFL} can be applied to semilinear elliptic equations involving the SFL and Dirichlet boundary conditions, in star-shaped domains $\Omega\subset \RN$:
	\begin{equation} \label{pb SFL dependence x}
		\begin{cases}
			(-\Delta|_{\Omega})^{s} u = f(x,u), & \mbox{in }\Omega,
			\\
			u = 0, & \mbox{on } \partial \Omega,
		\end{cases}
	\end{equation}
	where the non-linearity $f\in C^{0,1}_{loc}(\overline{\Omega}\times\R)$, and its primitive is defined as
	\begin{equation*}
		F(x,z)=\int_{0}^{z}f(x,t)\,dt\,.
	\end{equation*}
	The key idea is that, in star-shaped domains, $\Q 1[u]\geq 0$ and thus $\Q s[u]\geq 0$ by Theorem~\ref{Pohozaev SFL}.
	
	\begin{proposition}[\bf Pohozaev inequality for semilinear equations]\label{prop:Pohozaev-semilinear}
		Let $\Omega$ be a bounded $C^{1,1}$ star-shaped domain, $f\in C^{0,1}_{loc}(\overline{\Omega}\times\R)$ and $u$ be a bounded solution of \eqref{pb SFL dependence x}. Then,
		\begin{equation} \label{Pohozaev SFL dependence x}
			\lp\frac{2s-N}{2}\rp \int_{\Omega} uf(x,u)\,dx + N\int_{\Omega}F(x, u)\,dx + \int_{\Omega} x\cdot F_{x}(x, u) \, dx \geq 0\,.
		\end{equation}
	\end{proposition}
	
	As a consequence of the above inequality we obtain the following non-existence results:
	
	\begin{corollary}[\bf Non-existence of solutions I] \label{cor: non existence SFL dependence x} Under the assumptions of Proposition \ref{prop:Pohozaev-semilinear}, if
		\begin{equation} \label{ineq non existence SFL dependence x}
			\lp\frac{2s-N}{2}\rp tf(x, t) + NF(x,t) + x\cdot F_{x}(x,t) < 0, \quad \mbox{for all }x\in\Omega, \mbox{and } t\in \R\setminus\{0\},
		\end{equation}
		then problem \eqref{pb SFL dependence x} admits no non-trivial bounded solution.
	\end{corollary}
	
	If the non-linearity $f$ does not depend on $x$, then the above result simplifies to:
	
	\begin{corollary}[\bf Non-existence of solutions II] \label{cor: non-existence SFL}
		Under the assumptions of Proposition \ref{prop:Pohozaev-semilinear}, if
		\begin{equation}\label{condition-cor2}
			\int_{0}^{t}f(\tau)\,d\tau < \lp\frac{N-2s}{2N}\rp tf(t), \quad \mbox{for all }t\in \R\setminus\{0\},
		\end{equation}
		then problem \eqref{pb SFL dependence x} admits no non-trivial bounded solution.
	\end{corollary}
	
	A canonical application is the power non-linearity, $f(t)=|t|^{p-1}t$. In the subcritical range $1<p< \frac{N+2s}{N-2s}$, existence of non-trivial solutions can be easily shown by applying a variational mountain pass-argument thanks to the compact embedding $H_{0}^{s}(\Omega)\hookrightarrow L^{p+1}(\Omega)$ (see \cite{Servadei-Valdinoci} for the RFL; the argument for the SFL is analogous). For the supercritical case the above corollary yields:
	
	\begin{corollary}[\bf Non-existence of solutions III] \label{cor: non-existence semilinear pb} Let $\Omega$ be a bounded $C^{1,1}$ star-shaped domain and $p > \frac{N+2s}{N-2s}$. Then the problem
		\begin{equation*}
			\begin{cases}
				(-\Delta|_{\Omega})^{s} u =|u|^{p-1}u, & \mbox{in }\Omega,
				\\
				u = 0, & \mbox{on } \partial \Omega
			\end{cases}
		\end{equation*}
		admits no non-trivial bounded solution.
	\end{corollary}
	
	The critical value $p=\frac{N+2s}{N-2s}$ is quite delicate and Corollary \ref{cor: non-existence SFL} cannot be used, since equality in condition \eqref{condition-cor2} only implies that $\Q s[u]= 0$, which is insufficient a priori to establish non-existence. Whether non-trivial solutions exist at this critical exponent remains open, in contrast to the local case, where the Pohozaev identity combined with the Hopf lemma yields non-existence in star-shaped domains.
	
	%%%%%%%%%%%%%%%%%
	\subsection{The Spectral Fractional Laplacian and the Restricted Fractional Laplacian}
	
	In the study of diffusion processes, the classical heat equation with the Laplacian operator corresponds to Brownian motion, following Einstein's work \cite{Einstein}. When modeling phenomena involving Lévy flights—processes that allow for jumps of arbitrary length—one can subordinate the Brownian motion to obtain a jump process. The corresponding diffusion operator is a fractional power of the Laplacian, which gives rise to integro-differential equations; see \cite{RosOton2024IntDiffFernandezReal} for details on applications.
	
	In the Euclidean space, there are many different ways to define the Fractional Laplacian: through Fourier transform, through the heat semigroup and spectral analysis, using integral operators with hypersingular kernels, via the Caffarelli-Silvestre extension, etc. All these a priori different definitions of the Fractional Laplacian turn out to be equivalent in $\RN$. However, when dealing with bounded domains, most definitions give rise to different operators.
	
	Indeed, on the whole space, killing and subordination commute when applied to the Brownian motion. However, when dealing with bounded domains, the two operations do not commute and generate two different operators. By first killing (i.e. considering the Dirichlet Laplacian) and then subordinating (i.e. taking the spectral power), we obtain the SFL. For more details on the SFL and its applications to killed jump processes, see \cite{SV2003, SV2008}, and also \cite{AbTh, bonforte2018sharp, bonforte2017sharp, bonforte2014existence} and references therein.
	
	By first subordinating and then killing the process, we obtain the RFL. In order to evaluate this operator, one needs to know the values of the function outside $\Omega$, where it is set to zero. This justifies the terminology ``restricted'' and models the ``killed upon landing'' jump process. See \cite{RosOton2024IntDiffFernandezReal} and references therein for more details and applications.
	
	The SFL, already defined in Section \ref{section:main} via the spectral sum, admits an equivalent representation through the heat semigroup:
	\begin{equation*}
		(-\Delta|_{\Omega})^{s}u = \sum_{j=1}^{\infty}\lambda_{j}^{s}\hat{u}_{j}\varphi_{j} = \frac{1}{\Gamma(-s)}\int_{0}^{\infty} \lb e^{t\Delta|_{\Omega}}u-u \rb \frac{dt}{t^{1+s}},
	\end{equation*}
	where $\Gamma$ is the Gamma function. The SFL can also be represented using a hypersingular kernel:
	\begin{equation*}
		(-\Delta|_{\Omega})^{s}u(x) = \mbox{P.V.} \int_{\Omega} \lb u(x)-u(y) \rb J_{\Omega}(x,y)\,dy + K_{\Omega}(x) u(x)\,,
	\end{equation*}
	where $\mbox{P.V.}$ denotes the Cauchy principal value,
	\begin{equation*}
		K_{\Omega}(x)\asymp \dist x^{-2s},\qquad\mbox{with}\qquad \dist x:=\underset{x\in\Omega}{\min}\{|x-y|\,:\, y\in\partial\Omega\},
	\end{equation*}
	and we write $a \asymp b$ to denote that there exist constants $c, C>0$ such that $cb\leq a \leq Cb$. The jumping kernel $J_{\Omega}$ is a singular symmetric compactly supported kernel that degenerates at the boundary $\partial\Omega$. More precisely, we have the following estimates:
	\begin{equation}\label{kernel}
		J_{\Omega}(x,y) \asymp \frac{1}{|x-y|^{N+2s}} \, \min \lsb \frac{\dist x}{|x-y|},1\rsb \, \min \lsb \frac{\dist y}{|x-y|},1\rsb .
	\end{equation}
	We refer to \cite{SV2003,AbTh} for a proof of the above representation and estimates and for further details.
	
	In contrast, the RFL on domains shares the same kernel as the Fractional Laplacian on $\RN$, but is defined only for (restricted to) functions supported in $\overline{\Omega}$, more precisely:
	\begin{equation*}
		(-\Delta)^{s}u(x)=C_{N,s} \mbox{P.V.} \int_{\RN}\frac{u(x)-u(y)}{|x-y|^{N+2s}}\,dy,\qquad\mbox{whenever }{\rm supp}(u)\subseteq\overline{\Omega}\,,
	\end{equation*}
	where $C_{N,s}$ is a normalization constant given in \cite[Proposition A.1]{ros2014pohozaev}. Note that unlike the SFL kernel \eqref{kernel}, the RFL kernel is translation and rotation invariant, and homogeneous.
	
	Finally, both operators admit representations as Dirichlet-to-Neumann operators for extension problems. For the RFL, this is the Caffarelli-Silvestre extension \cite{Caffarelli-Silvestre}. For the SFL, such extensions were constructed by Cabr\'{e} and Tan \cite{Cabre-Tan} for the half-Laplacian, and by Br\"andle, Colorado, de Pablo and S\'{a}nchez \cite{Brandle-Colorado-dePablo-Sanchez}.
	
	%%%%%%%%%%%%%%%%%%%%%
	\subsection{Known results on the Pohozaev identity}
	
	The Pohozaev identity in the local case was first proved by Stanislav I. Pohozaev in his 1965 paper \cite{pohozaev1965eigenfunctions},
	
	\begin{theorem}\label{Thm Pohozaev local} \rm (\textbf{Pohozaev identity for the Laplacian} \cite[Theorem 5.1]{quittner2019superlinear})
		Let $N\geq 3$ and let $u\in C^{2}(\overline{\Omega})$ satisfy $u=0$ on $\partial\Omega$. Then
		\begin{equation*}
			\Q 1[u]:=\lp\frac{2-N}{2}\rp \int_{\Omega} u(-\Delta) u\,dx - \int_{\Omega}(x\cdot \nabla u)(-\Delta)u\,dx = \frac{1}{2} \int_{\partial \Omega}\left|\frac{\partial u}{\partial\nu}\right|^{2}(x \cdot \nu)\,dS,
		\end{equation*}
		where $\nu$ is the outward normal and $dS$ is the surface measure. If $\Omega$ is star-shaped, $\Q 1[u]\ge 0$.
	\end{theorem}
	
	For the classical Laplacian, the term $\frac{\partial u}{\partial \nu}$ in the identity arises naturally from integration by parts and plays a fundamental role in relating the behavior of the solution inside the domain to its boundary conditions.
	
	The spectral half-Laplacian can be expressed as the Dirichlet-to-Neumann operator for a harmonic extension with Dirichlet conditions on the half cylinder $\partial\Omega\times \mathbb R_+$. It therefore inherits a Pohozaev identity from that of the Laplacian in $\Omega\times\mathbb R_+$, as shown by Tan \cite{Tan-BN problem} for the half-SFL. However, the boundary term in this identity is defined on $\partial\Omega\times\mathbb R_+$ rather than just on $\partial\Omega$, since it depends on the extension.
	
	In the genuine non-local setting, the first formulation of the Pohozaev identity was obtained by Xavier Ros-Oton and Joaquim Serra in their celebrated 2014 article \cite{ros2014pohozaev} for the RFL, see formula \eqref{Pohozaev-ROS} below. Let us state the main result of \cite{ros2014pohozaev} in detail, to highlight the role of boundary regularity. %Here, $\frac{u}{\rd^{s}}\big|_{\partial\Omega}$ plays the role of $\frac{\partial u}{\partial \nu}$ for the RFL.
	
	\begin{theorem}[Proposition 1.6 of \cite{ros2014pohozaev}]\label{theorem:RS}
		Let $\Omega$ be a bounded $C^{1,1}$ domain. Assume that $u$ is a $H^{s}\lp\RN\rp$ function which vanishes in $\RN \backslash \Omega$, and satisfies
		\begin{enumerate}
			\item $u \in C^{s}\lp\RN\rp$ and, for every $\beta \in[s, 1+2 s)$, $u$ is of class $C^{\beta}(\Omega)$ and
			\begin{equation*}
				[u]_{C^{\beta}(\{x \in \Omega\,|\, \dist x \geqq \rho\})} \leqq C \rho^{s-\beta}, \qquad \mbox {for all } \rho \in(0,1).
			\end{equation*}
			\item The function $\frac{u}{\rd^{s}}\big|_{\Omega}$ can be continuously extended to $\overline{\Omega}$. Moreover, there exists $\alpha \in(0,1)$ such that $\frac{u}{\rd^{s}} \in C^{\alpha}(\overline{\Omega})$. In addition, for all $\beta \in[\alpha, s+\alpha]$, the following estimate holds
			\begin{equation*}
				\lb \frac{u}{\rd^{s}} \rb_{C^{\beta}(\{x \in \Omega\, |\, \dist x \geqq \rho\})} \leqq C \rho^{\alpha-\beta}, \qquad \mbox{for all } \rho \in(0,1).
			\end{equation*}
			\item $(-\Delta)^{s} u$ is pointwise bounded in $\Omega$.
		\end{enumerate}
		Then, the following identity holds
		\begin{equation}\label{Pohozaev-ROS}
			\lp\frac{2s-N}{2}\rp \int_{\Omega} u(-\Delta)^{s} u \,dx - \int_{\Omega}(x \cdot \nabla u)(-\Delta)^{s} u \, dx = \frac{\Gamma(1+s)^{2}}{2} \int_{\partial \Omega}\lp\frac{u}{\rd^{s}}\rp^{2}(x \cdot \nu) \, dS.
		\end{equation}
	\end{theorem}
	%
	%\begin{proof}
	%See Proposition 1.6 in \cite{ros2014pohozaev}.
	%\end{proof}
	
	Ros-Oton and Serra showed in \cite{ros2014dirichlet} that solutions of semilinear problems for the RFL indeed have the right regularity to apply Theorem \ref{theorem:RS}. More precisely:
	
	\begin{theorem}[Theorem 1.1 of \cite{ros2014pohozaev}]
		Let $\Omega$ be a bounded $C^{1,1}$ domain, $f$ be a locally Lipschitz function and $u$ be a bounded solution of
		\begin{equation*}
			\begin{cases}
				(-\Delta)^{s} u = f(u), & \mbox{in }\Omega,
				\\
				u = 0, & \mbox{in } \RN\setminus \Omega.
			\end{cases}
		\end{equation*}
		Then $\frac{u}{\rd^{s}}\big|_{\Omega}\in C^{\alpha}(\overline{\Omega})$, for some $\alpha\in(0,1)$, meaning that $\frac{u}{\rd^{s}}\big|_{\Omega}$ has a H\"older continuous extension to $\overline{\Omega}$, and the following identity holds
		\begin{equation*}
			\lp\frac{2s-N}{2}\rp\int_{\Omega} u f(u)\,dx + N \int_{\Omega} F(u)\,dx = \frac{\Gamma(1+s)^{2}}{2}\int_{\partial\Omega} \lp\frac{u}{\rd^{s}}\rp^{2}(x\cdot \nu)\, dS,
		\end{equation*}
		where $F(z)=\int_{0}^{z}f(t)\,dt$.
	\end{theorem}
	
	In their original proof in \cite{ros2014pohozaev}, Ros-Oton and Serra used the regularity of the solution up to the boundary to obtain a boundary term involving $u/\rd^{s}$. This term serves as the non-local counterpart for $\frac{\partial u}{\partial \nu}$ in the local case.
	Nevertheless, in our setting solutions are Lipschitz continuous and the approach in \cite{ros2014pohozaev} does not recover any boundary term.
	
	%\color{darkgreen}Podriamos mencionar aqui tambien el resultado del CFL, diciendo que es igual a lo de Xavi y Joaquim (a lo mejor solo poner uan formula y la referencia), decir que solo cambia la regularidad al borde (el nucleo no es degenerado, se puede repetir la prueba de RO-S) y que vale entre $1/2<s<1$. Así hacemos la diferencia entre los 3 Laplacianos. \normalcolor

	%%%%%%%%%%%%%%%%%
	\subsection{Why the Ros-Oton–Serra proof strategy fails for the SFL}
	
	The SFL differs fundamentally from the RFL in ways that prevent adapting the Ros-Oton–Serra approach \cite{ros2014pohozaev}. First, the SFL lacks the translation and dilation invariance that the RFL inherits from the fractional Laplacian on $\RN$; these symmetries play a crucial role in \cite{ros2014pohozaev}. Second, the boundary regularity differs markedly between the two operators. The main reason for this difference is that the kernel of the SFL degenerates at the boundary (recall \eqref{kernel}), while the RFL kernel does not. See \cite[Sections 2A and 10A]{bonforte2017sharp}, \cite{bonforte2018sharp} and references therein for further details on boundary estimates and regularity theory.
	
	In the semilinear case with power non-linearity $|u|^{p-1}u$, the boundary behaviour of $u$ is dictated by the first eigenfunction:
	$u(x)\asymp {\rm d}(x)$. This is in stark contrast to the RFL case, where $u\in C^s(\overline{\Omega})$ and $u(x)\asymp {\rm d}^{s}(x)$.
	
	To provide broader context, the RFL and SFL also exhibit distinct boundary behavior in the linear setting. Solutions to the Dirichlet problem $Lu=f$ are bounded when $f\in L^p(\Omega)$ with $p>N/2s$, both for $L = (-\Delta)^s$ (RFL) and for $L = (-\Delta|_{\Omega})^s$ (SFL). Bounded solutions of the RFL enjoy optimal $C^{s}(\overline{\Omega})$ regularity (see \cite{ros2014dirichlet}), and this holds in particular for the eigenfunctions, which also satisfy the boundary estimates $ \varphi_1(x)\asymp {\rm d}^s(x)$ and $| \varphi_k(x)|\lesssim {\rm d}^s(x)$. On the other hand, the eigenfunctions of the SFL are the same as those of the classical Dirichlet Laplacian. Hence they can be as regular as the domain allows (up to $C^\infty$), and in particular they are at least Lipschitz in $C^{1,1}$ domains, with $ \varphi_1(x)\asymp {\rm d}(x)$ and $| \varphi_k(x)|\lesssim {\rm d}(x)$.
	
	When $s>1/2$, the first eigenfunction dominates the boundary behaviour of bounded solutions, as in the case of the RFL. However, when $s\le 1/2$, the boundary behaviour may differ due to the degeneracy of the kernel and the nature of the forcing term. To illustrate this, consider three representative cases:
	
	\begin{itemize}
		\item The eigenvalue problem. The solution to $(-\Delta|_{\Omega})^{s}u = \lambda u$ with $u=0$ on $\partial\Omega$ satisfies $u(x)\asymp {\rm d}(x)$ for all $s\in(0,1)$.
		
		\item The case $f=0$ at the boundary. When the forcing term vanishes at the boundary, solutions inherit the behaviour of the eigenfunctions: $u(x)\asymp {\rm d}(x)$.
		
		\item The case $f=1$ (stopping time problem). When the forcing term does not vanish at the boundary, the situation changes for $s\le 1/2$. Solutions behave like ${\rm d}(x)(1 + |\log({\rm d}(x))|)$ when $s=1/2$ and as ${\rm d}(x)^{2s}$ when $s\in (0,1/2)$.
	\end{itemize}
	
	The boundary behaviour and regularity of bounded solutions of the SFL when $s\le 1/2$ is a delicate issue that remains only partially understood. These differences in boundary behaviour are one of the main obstacles to proving Pohozaev identity for the SFL following the approach of Ros-Oton and Serra in \cite{ros2014dirichlet,ros2014pohozaev}. They also raise the intriguing question of optimal boundary regularity for the SFL when $s\le 1/2$.
	
	%%%%%%%%%%%%%%%%%%%%
	\subsection{A bit of history}
	
	In 1943, Rellich \cite{rellich1943asymptotische} established a preliminary version of \eqref{eq: abstract Pohozaev local} for the particular case ${f = \frac{1}{2}\lambda u^{2}}$. Later, in 1961, Morawetz used similar techniques in spirit to Pohozaev’s method to get integral identities and inequalities for the wave equation \cite{morawetz1961decay}.
	
	In 1965, Pohozaev \cite{pohozaev1965eigenfunctions} discovered the identity for the standard Laplacian which now carries his name, providing a fundamental tool in the study of semilinear elliptic problems and, in particular, allowing one to prove sharp non-existence results in supercritical semilinear problems. Some classical applications were mentioned above and include the non-existence results by Pucci and Serrin \cite{Pucci-Serrin-variational} and the study of the perturbed problem by Br\'ezis and Nirenberg \cite{BrezisPositiveSolu1983Nirenberg} in star-shaped domains (see also the earlier work by Moser and Rabinowitz \cite{Moser1974VariationalRabinowitz}).
	
	Since then, a number of generalizations have been obtained in the local setting. For instance, in 2010 Bozhkov and Mitidieri developed extensions of the Rellich-Pohozaev identities using Noether-type approximation techniques \cite{boz2010Noether}. For systems, Van der Vorst \cite{VanderVorst} generalized Pucci-Serrin's results. Still in the elliptic framework, Pohozaev identities can show uniqueness for problems where a non-trivial solution exists \cite{Dolbeault-Stanczy}. They also may be used to study blow-up behavior near critical exponents (Br\'{e}zis-Peletier conjecture) \cite{Brezis-Peletier,Han-asymptotic,Rey}. Additionally, in the paper \cite{Serra:radial-symmetry}, Serra used the method of Lions in \cite{PLLions} that combines an isoperimetric inequality with a Pohozaev identity in order to prove radial symmetry of solutions and obtain an alternative proof to the classical result by Gidas-Ni-Nirenberg, see \cite{Gidas-Ni-Nirenberg}. More general elliptic PDEs with associated Pohozaev identities include models for nematic liquid crystals \cite{ChouZhu} and the anisotropic $p$-Laplacian, already considered by Pucci and Serrin in \cite{Pucci-Serrin-variational}.
	
	In geometric analysis, Pohozaev-type identities arise as conservation laws related to geometric structures \cite{Riviere}. Notable examples include the Pohozaev-Schoen formula in \cite{Schoen-existence}, written in terms of a conformal killing vector field on a manifold with boundary, which is used in the study of the Yamabe problem in conformal geometry. Pohozaev-type identities have also been considered in the context of harmonic maps, exploiting the scaling properties of the stress-energy tensor \cite{Schoen:harmonic-map, Schoen-Uhlenbeck}.
	%MIRAR LIBRO DE JOST.
	
	Beyond elliptic equations, in the study of dispersive PDE, Pohozaev-type identities are known as Morawetz identities (following the work of Morawetz \cite{morawetz1961decay,Morawetz2}) or, more generally, virial identities that appear due to invariance under scaling (see, for example, the recent lecture notes on the non-linear Schr\"odinger equation \cite{Cazenave-Schrodinger} for a thorough account of the subject).
	
	The spectral half-Laplacian in $\Omega$ may be defined using the Caffarelli-Silvestre extension of \cite{Caffarelli-Silvestre} (see also \cite{Brandle-Colorado-dePablo-Sanchez,Cabre-Tan}) with Dirichlet conditions on $\partial\Omega$. As the associated extension is harmonic in $\Omega\times\mathbb R_+$, the half-Laplacian inherits a Pohozaev identity from that of the Laplacian, written in terms of the extension variable, as proved by Tan in \cite{Tan-BN problem} in 2011. Note that the boundary integral in this expression is defined on the whole cylinder $\partial\Omega\times \mathbb R_+$.
	
	The first Pohozaev identity that was genuinely non-local (meaning, with a boundary term defined only on $\partial\Omega$ and not on the extension) was derived by Ros-Oton and Serra in 2014 in their famous paper \cite{ros2014pohozaev} for the RFL, proving non-existence results for the associated semilinear PDE. Afterwards, Ros-Oton, Serra, and Valdinoci generalized their method to anisotropic integro-differential operators, see \cite{ros2017pohozaev}, and the survey \cite{RosOton:survey}.
	
	In the case of the whole $\mathbb R^N$, where there is no boundary term, an integration by parts procedure yields a Pohozaev formula. This method has been applied to derive a Pohozaev identity for the fractional
	anisotropic $p-$Laplacian in $\RN$, see \cite{ambrosio2024plaplacian,ambrosio2025fractional}. Another application of Pohozaev-type energy identities appears in the study of half-harmonic maps (see the note \cite{DaLio} and the references therein).
	
	Finally, in 2016, Grubb obtained Pohozaev identities for space-dependent fractional-order operators using microlocal analysis techniques \cite{grubb2016integration}. Her proof also recovers the original result by Ros-Oton and Serra \cite{ros2014pohozaev}.

	%%%%%%%%%%%%%%%%%%%%%%%%
	\section{Proof of the Pohozaev identity for the SFL}
	
	We recall two key results that are essential for the proof of Theorem \ref{Pohozaev SFL}: the Schur product theorem and Bochner's theorem.
	
	\begin{theorem}\rm (\textbf{Schur product theorem} \cite[Theorem 7.5.3]{horn2012matrix})\textit{\label{Schur} Let $A=[a_{ij}], B=[b_{ij}] \in M_{n}$ be two positive semidefinite matrices, where $M_{n}$ is the set of $n \times n$ complex matrices. Define the Schur (or Hadamard) product as $A \circ B = [a_{ij}b_{ij}]$. Then $A \circ B$ is positive semidefinite.}
		%Then the following assertions are true:
		%\begin{enumerate}
		%    \item $A\circ B$ is positive semidefinite.
		%    \item If $A$ is positive definite and every main diagonal entry of $B$ %is positive, then $A\circ B$ is positive definite.
		%    \item If both $A$ and $B$ are positive definite, then $A\circ B$ is %positive definite.
		%\end{enumerate}
	\end{theorem}
	
	The following celebrated result characterizes positive semidefinite functions:
	
	\begin{theorem}\rm (\textbf{Bochner theorem} \cite[Sections 1.4.1 and 1.4.3]{rudin2017fourier})\label{Bochner}\textit{ A continuous function $H:\RN\xrightarrow{}\R$ is positive semidefinite if and only if it is the Fourier transform of a non-negative bounded regular Borel measure $\mu$ on $\RN$. Equivalently, for every choice of complex numbers $c_{1},\dots, c_{n}$ and any finite set of points $x_{1}, \dots, x_{n}$, we have
			\begin{equation*}
				\sum_{j,k=1}^{n} c_{j} \overline{c}_{k} H(x_{j}-x_{k}) \geq 0
			\end{equation*}
			if and only if there exists a non-negative bounded regular Borel measure $\mu$ on $\RN$ such that
			\begin{equation*}
				H(t) = \int_{\RN} e^{-2\pi i \langle t,x\rangle} \, d\mu(x),
		\end{equation*}}
		for all $t \in \RN$.
	\end{theorem}
	
	%%%%%%%%%%%%%%%%%%
	Now let us proceed with the proof of the Pohozaev identity for the SFL.
	
	\begin{remark2}[\textit{On the regularity hypotheses of Theorem \ref{Pohozaev SFL}}]
		The $C^{1,1}$ assumption on the domain makes the proof clearer by avoiding heavy technicalities, though it may be possible to relax this condition. The regularity $u\in H^{\max \lsb 2s, 3/2 \rsb}_0(\Omega)$ is needed in the decomposition of $\Q s$ as a Schur product of $\p s$ and $\Q 1$ to ensure that all the quantities involved are finite. On the one hand, the exponent $3/2$ is the minimal regularity requirement for the boundedness of $\Q 1$, when written in the form of a boundary integral as in \eqref{eq: abstract Pohozaev local}. On the other hand, the exponent $2s$ is required to have $(-\Delta|_{\Omega})^{s}u\in L^2(\Omega)$. Notably, when $s\in [3/4,1]$, the $H_0^{2s}$ regularity is sufficient for all such purposes.
		However, if we only need $\Q s\ge 0$, we can show by approximation that the sign is preserved even if $\Q 1$ diverges (see the proof of Proposition \ref{prop:Pohozaev-semilinear}).
	\end{remark2}
	\medskip
	
	\begin{proof}[\bf Proof of Theorem \ref{Pohozaev SFL}] We shall prove first the quadratic structure in the $s$-Pohozaev expression \eqref{Pohozaev SFL id}, that we recall here:
		\begin{equation} \label{Pohozaev SFL.proof}
			\Q s [u] :=\lp\frac{2s-N}{2}\rp \int_{\Omega} u(-\Delta|_{\Omega})^{s} u\,dx - \int_{\Omega}(x \cdot \nabla u)(-\Delta|_{\Omega})^{s} u \,dx.
		\end{equation}
		We shall see that the right-hand side above can be rewritten as a quadratic form over our chosen Fourier basis $(\lambda_{k}, \varphi_{k})$ of $L^2(\Omega)$,
		\begin{equation}\label{definition-Q}
			\Q s [u] = \suma{j,k=1}{\infty} \hQ s j k \hat{u}_{j} \hat{u}_{k} = \suma{j,k=1}{\infty} \mP s j k \hQ 1 j k \hat{u}_{j} \hat{u}_{k}, %= s\sum_{j=1}^{\infty} \hQ 1 j j\lambda_{j}^{s}\hat{u}_{j}^{2} + 2\suma{j,k=1}{\infty}\frac{\lambda_{k}^{s}\sqrt{\lambda_{j}\lambda_{k}}}{(\lambda_{k}-\lambda_{j})}\hat{u}_{j}\hat{u}_{k}\hQ 1 j k,
		\end{equation}
		%We remember that $\hQ 1 j k = \frac{1}{2} \int_{\partial \Omega} (\nabla \varphi_{j} \cdot \nabla \varphi_{k})(x \cdot \nu)\,dS$, and, $\mP s j k = \lp\frac{\lambda_{j}^{s}-\lambda_{k}^{s}}{\lambda_{j}-\lambda_{k}}\rp$. This can be rewritten as
		%
		where the transition matrix $\mP s {}{}$ is given in \eqref{formula-P}, and $\Q 1 [u]$ is the quadratic form from the classical Pohozaev identity with the expression
		\begin{equation*}
			\Q 1[u]=\suma{j,k=1}{\infty} \hQ 1jk \hat{u}_{j} \hat{u}_{k}, \quad \mbox{where }
			\hQ 1 j k = \frac{1}{2} \int_{\partial \Omega} (\nabla \varphi_{j} \cdot \nabla \varphi_{k})(x \cdot \nu)\,dS.
		\end{equation*}
		The second part of the proof will focus on showing that $\Q 1 [u] \ge 0$ implies that $\Q s [u]\ge 0$. In order to simplify the presentation, we will assume that all eigenvalues $\lambda_{j}$ are different, and postpone the case of multiplicity until the end of the proof.
		
		\medskip
		%%%%%%%%%%%%%%%%%
		\noindent\textsc{Step 1.} \textit{A consequence of the classical Pohozaev identity in Fourier variables.} We first prove the following preliminary identity, valid for all $j,k\ge 1$,
		\begin{equation}\label{consequence classical Pohozaev}
			\int_{\Omega}\lp x \cdot \nabla \varphi_{j}\rp \varphi_{k}\, dx=
			\begin{cases}
				\displaystyle -N/2, & \mbox{for }j=k,
				\\[0.1cm]
				\displaystyle\frac{1}{\lambda_{j}-\lambda_{k}} \int_{\partial\Omega} \nabla \varphi_{j} \cdot \nabla \varphi_{k}(x \cdot \nu) \, dS, & \mbox{for } j\neq k \mbox{ with }\lambda_{j}\neq\lambda_{k} .
			\end{cases}
		\end{equation}
		For $j = k$, we have
		$\nabla(\varphi_{j}^{2})=2\varphi_{j} \nabla \varphi_{j} $, and then
		\begin{equation}\label{plugging eigenfunction}
			\int_{\Omega}\lp x \cdot \nabla \varphi_{j}\rp \varphi_{j}\,dx=\frac{1}{2} \int_{\Omega} x \cdot \nabla\lp\varphi_{j}^{2}\rp\,dx=-\frac{1}{2} \int_{\Omega} \varphi_{j}^{2} \,\nabla\cdot x \,dx=-\frac{N}{2}.
		\end{equation}
		Now we consider the case $j\ne k$, for which we shall use the classical Pohozaev identity:
		\begin{equation} \label{Pohozaev local.proof}
			\frac{(2-N)}{2} \int_{\Omega} u(-\Delta) u\,dx - \int_{\Omega}(x\cdot \nabla u)(-\Delta)u\,dx = \frac{1}{2} \int_{\partial \Omega}\left|\frac{\partial u}{\partial\nu}\right|^{2}(x \cdot \nu)\,dS.
		\end{equation}
		Substituting $u = \varphi_{j}$ and using \eqref{plugging eigenfunction} in the above formula gives, for all $j\ge 1$,
		\begin{equation*}
			\lambda_{j}=
			%\frac{2-N}{2} \lambda_{k}+\frac{N}{2} \lambda_{k}=
			\frac{1}{2} \int_{\partial \Omega}|\nabla \varphi_{j}|^{2}(x \cdot \nu) \, dS\,.
		\end{equation*}
		With this in mind, it will be useful to write
		\begin{equation}\label{Q1}
			\hQ 1 j k = \frac{1}{2} \int_{\partial \Omega} (\nabla \varphi_{j} \cdot \nabla \varphi_{k})(x \cdot \nu)\,dS=\lambda_{k} \delta_{jk}
			+\frac{(1-\delta_{jk})}{2}\int_{\partial\Omega} (\nabla \varphi_{j} \cdot \nabla \varphi_{k})(x \cdot \nu)\,dS.
		\end{equation}
		Next, we substitute $u = \varphi_{j}$, $u = \varphi_{k}$ and $u = \varphi_{j}+\varphi_{k}$ into \eqref{Pohozaev local.proof}, which yields
		\begin{equation}\label{plugging eigenfunctions}
			\lambda_{k} \int_{\Omega}\lp x \cdot \nabla \varphi_{j}\rp \varphi_{k} \, dx+ \lambda_{j} \int_{\Omega}\lp x \cdot \nabla \varphi_{k}\rp \varphi_{j} \,dx=-\int_{\partial\Omega} \lp\nabla \varphi_{j} \cdot \nabla \varphi_{k}\rp\,(x \cdot \nu)\,dS,
		\end{equation}
		taking into account the identities above. In addition,
		\begin{equation}\label{antisymmetry.1}
			\begin{aligned}
				\int_{\Omega}\lp x \cdot \nabla \varphi_{j}\rp \varphi_{k} \, dx + \int_{\Omega}\lp x \cdot \nabla \varphi_{k}\rp \varphi_{j}\, dx = & \int_{\Omega}\lp x \cdot \nabla\lp\varphi_{j} \varphi_{k}\rp\rp\, dx
				\\
				= & \int_{\Omega} \varphi_{j} \varphi_{k} \nabla\cdot x \, dx =0,
			\end{aligned}
		\end{equation}
		since $\varphi_{j}$ and $\varphi_{k}$ are orthogonal.
		
		\noindent If $\lambda_{j}\neq \lambda_{k}$ we obtain from \eqref{plugging eigenfunctions} that
		\begin{equation}\label{consequence classical Pohozaev with lambdas}
			(\lambda_{j}-\lambda_{k})\int_{\Omega}\lp x \cdot \nabla \varphi_{j}\rp \varphi_{k} \, dx=\int_{\partial\Omega} \lp \nabla \varphi_{j} \cdot \nabla \varphi_{k}\rp\,(x \cdot \nu)\,dS,
		\end{equation}
		hence,
		\begin{equation*}
			\int_{\Omega} \lp x \cdot \nabla \varphi_{j}\rp \varphi_{k} \,dx = \frac{1}{\lambda_{j}-\lambda_{k}} \int_{\partial\Omega} \lp\nabla \varphi_{j} \cdot \nabla \varphi_{k}\rp\,(x \cdot \nu)\,dS\,,\quad \mbox{for }\lambda_{j}\neq\lambda_{k},
		\end{equation*}
		and this concludes the proof of \eqref{consequence classical Pohozaev}.
		
		%%%%%%%%%%%%%%%%%
		\noindent\textsc{Step 2.} \textit{Rewriting the right-hand side of \eqref{Pohozaev SFL.proof} in Fourier variables.} We have that
		\begin{equation}\label{Pohozaev in Fourier}\begin{split}
				\lp \frac{2s-N}{2} \rp \int_{\Omega} u(-\Delta|_{\Omega})^{s} u\,dx & - \int_{\Omega}(x \cdot \nabla u)(-\Delta|_{\Omega})^{s} u \,dx\\
				&= \lp \frac{2s-N}{2} \rp \sum_{j=1}^{\infty} \lambda_{j}^{s} \hat{u}_{j}^{2}
				- \suma{j,k=1}{\infty} \hat{u}_{j} \hat{u}_{k} \lambda_{k}^{s} \int_{\Omega}\lp x \cdot \nabla \varphi_{j}\rp \varphi_{k}\,dx.
			\end{split}
		\end{equation}
		To see this, recall that $u=\sum_{k=1}^{\infty} \hat{u}_{k} \varphi_{k}$. Then the first integral in \eqref{Pohozaev in Fourier} becomes
		\begin{equation*}
			\int_{\Omega} u\lp-\Delta| _{\Omega}\rp^{s} u \, dx= \suma{j,k=1}{\infty} \hat{u}_{k} \hat{u}_{j} \int_{\Omega} \varphi_{k} \lambda_{j}^{s} \varphi_{j} \, dx=\sum_{j=1}^{\infty} \lambda_{j}^{s} \hat{u}_{j}^{2},
		\end{equation*}
		and for the second integral in \eqref{Pohozaev in Fourier}, we find
		\begin{equation*}
			\int_{\Omega}(x \cdot \nabla u)\lp-\Delta|_{\Omega}\rp^{s} u \,dx
			= \suma{j,k=1}{\infty} \hat{u}_{j} \hat{u}_{k} \lambda_{k}^{s} \int_{\Omega}\lp x \cdot \nabla \varphi_{j}\rp \varphi_{k}\,dx.
		\end{equation*}
		
		%%%%%%%%%%%%%%%%
		\noindent\textsc{Step 3.} \textit{The quadratic form structure of the $s$-Pohozaev identity in Fourier variables. } Our aim here is to show that the right-hand side of the Pohozaev identity, written in \eqref{Pohozaev in Fourier} in Fourier variables, has the quadratic structure given in \eqref{definition-Q}. That is, $\Q s$ can be written as a Schur product, more precisely:
		\begin{equation}\label{Schur-structure}\Q s [u]=\p s\circ \Q 1[u],
		\end{equation}
		where $\Q 1$ is defined in \eqref{Q1}
		and we recall the formula for $\mP s {}{}$ from \eqref{formula-P}:
		\begin{equation}\label{Ps}
			\mP s j k = s\lambda_{j}^{s-1}\delta_{jk} + \lp\frac{\lambda_{j}^{s}-\lambda_{k}^{s}}{\lambda_{j}-\lambda_{k}}\rp(1-\delta_{jk}).
		\end{equation}
		
		\noindent Now, the proof of \eqref{Schur-structure}--\eqref{Ps} proceeds as follows: we have seen in \eqref{Pohozaev in Fourier} that the right-hand side of the $s$-Pohozaev identity can be written in Fourier variables as
		\begin{equation*}
			\Q s [u] = \lp\frac{2s-N}{2}\rp \sum_{j=1}^{\infty} \lambda_{j}^{s} \hat{u}_{j}^{2}
			- \suma{j,k=1}{\infty} \lambda_{k}^{s} \hat{u}_{j} \hat{u}_{k} \int_{\Omega}\lp x \cdot \nabla \varphi_{j}\rp \varphi_{k}\,dx =: \suma{j,k=1}{\infty} \tilde{Q}^{(s)}_{j k} \hat{u}_{j} \hat{u}_{k},
		\end{equation*}
		where the coefficients of the (non-symmetric) quadratic form $\tilde{Q}^{(s)}$ are given by
		\begin{equation}\label{Q tilde}
			\tilde{Q}^{(s)}_{j k}=\lp\frac{2s-N}{2}\rp \lambda_{k}^{s} \delta_{jk}
			- \lambda_{k}^{s}\int_{\Omega} \lp x \cdot \nabla \varphi_{j}\rp \varphi_{k}\,dx.
		\end{equation}
		Applying \eqref{consequence classical Pohozaev} we get
		\begin{equation*}
			\tilde{Q}^{(s)}_{j k}=
			s \lambda_{j}^{s} \delta_{j k}- (1-\delta_{jk})\displaystyle\frac{\lambda_{k}^{s}}{\lambda_{j}-\lambda_{k}} \int_{\partial \Omega} \nabla \varphi_{j} \cdot \nabla \varphi_{k}(x \cdot \nu)\,dS.
		\end{equation*}
		We symmetrize the matrix by defining $\hQ s j k=\frac{1}{2}\lp\tilde{Q}^{(s)}_{j k}+\tilde{Q}^{(s)}_{k j}\rp$, which does not change the quadratic form. Thus
		\begin{equation*}
			\Q s [u] = \suma{j,k=1}{\infty} \tilde{Q}^{(s)}_{j k} \hat{u}_{j} \hat{u}_{k} = \suma{j,k=1}{\infty} \hQ s jk \hat{u}_{j} \hat{u}_{k},
		\end{equation*}
		where
		\begin{equation*}
			\hQ s jk=
			s \lambda_{j}^{s} \delta_{j k}+\frac{\lp1-\delta_{j k}\rp}{2}\lp\frac{\lambda_{j}^{s}-\lambda_{k}^{s}}{\lambda_{j}-\lambda_{k}}\rp\int_{\partial \Omega} (\nabla \varphi_{j} \cdot \nabla \varphi_{k})(x \cdot \nu)\,dS,
		\end{equation*}
		and \eqref{Schur-structure} follows by identifying terms in \eqref{Q1}.
		
		%%%%%%%%%%%%%%%%%
		\noindent \textsc{Step 4.} \textit{Positivity of $\Q s$ via Schur's product theorem.}
		We apply Theorem \ref{Schur} to the Schur product \eqref{Schur-structure}. We first establish that $\Q 1$ is positive semidefinite. In star-shaped domains, we have
		\begin{equation*}
			\suma{j,k=1}{\infty} \hQ 1 j k \hat{u}_{j} \hat{u}_{k} = \frac{1}{2} \int_{\partial \Omega}\left|\sum_{j=1}^{\infty} \nabla \varphi_{j} \hat{u}_{j} \right|^2(x \cdot \nu)\,dS\geq 0, \qquad \mbox{for all } u\in H^{3/2}(\overline{\Omega}).
		\end{equation*}
		To apply Theorem \ref{Schur}, we consider the finite-dimensional matrices $\Qn s := \Pn s \circ \Qn 1$, where $\Pn s$ and $\Qn 1$ are the $n\times n$ truncations of $\p s$ and $\Q 1$, respectively. Since $\Qn 1 \geq 0$, Theorem \ref{Schur} will give $\Qn s \geq 0$ for all $n$ once we establish that $\Pn s \geq 0$. We will prove this in the next step. Since $u \in H^{\max\{2s,3/2\}}$, we can pass to the limit $n\to\infty$ and obtain $\Q s \geq 0$.
		
		%%%%%%%%%%%%%%%%%%
		\noindent \textsc{Step 5.} \textit{$\Pn s$ is positive semidefinite through Bochner's theorem.}
		Recall that $\Pn s$ is the $n\times n$ truncation of the infinite matrix $\p s$ with elements given by \eqref{Ps}. To show $\Pn s \geq 0$ for all $n$, it suffices to prove that $\p s$ is positive semidefinite.
		
		Writing $\lambda_{k}=e^{2 \mu_{k}}$, we have
		\begin{equation*}
			\begin{aligned}
				\mP s j k & =  se^{2 (s-1)\mu_{j}}\delta_{jk} + \frac{e^{2s\mu_j} - e^{2s\mu_k}}{e^{2\mu_j} - e^{2\mu_k}} \lp 1-\delta_{jk}\rp
				\\
				& = se^{(s-1)(\mu_{j}+\mu_{k})}\delta_{jk} + e^{(s-1)(\mu_{j}+\mu_{k})}\frac{e^{s(\mu_j-\mu_k)} - e^{-s(\mu_j-\mu_k)}}{e^{\mu_j - \mu_k} - e^{-(\mu_j-\mu_k)}}\lp 1-\delta_{jk}\rp
				\\
				& = e^{(s-1)(\mu_{j}+\mu_{k})} \tilde{P}_{jk}^{s},
			\end{aligned}
		\end{equation*}
		where
		\begin{equation*}
			\tilde{P}_{jk}^{s}= s\delta_{jk} + \frac{\sinh \lp s \lp \mu_{j}-\mu_{k} \rp \rp}{\sinh \lp \mu_{j}-\mu_{k} \rp}\lp 1-\delta_{jk}\rp=:H_s \lp \mu_{j}-\mu_{k} \rp,
		\end{equation*}
		and $H_s(t) = \frac{\sinh(s t)}{\sinh t}$, with the understanding that $H_s(0) = s$ by continuity.
		
		In order to prove the positivity of $P^{(s)}$, we observe that $\mP s jk$ can be written as a Schur product $e^{(s-1)(\mu_{j}+\mu_{k})} \circ \tilde{P}_{jk}^{s}$. Applying Theorem \ref{Schur}, it suffices to show that both factors are positive semidefinite matrices. The matrix with entries $e^{(s-1)(\mu_j + \mu_k)}$ is positive semidefinite as an outer product of a vector with itself. Thus it remains to show that $\tilde{P}^s$ is positive semidefinite.
		
		For this, we apply Theorem \ref{Bochner} which says that $H_s$ is a positive semidefinite function if and only if $H_s$ is the Fourier transform of a non-negative measure. Thus we need to calculate
		\begin{equation*}
			\begin{aligned}
				\mathcal{F}^{-1}(H_s)(\xi) = & \int_{\R} e^{2\pi i \xi t}H_s(t)\,d t
				\\
				= & \int_{\R} \cos(2\pi \xi t) \frac{\sinh(s t)}{\sinh t}\,d t + i\int_{\R} \sin (2\pi \xi t) \frac{\sinh(s t)}{\sinh t }\,d t.
			\end{aligned}
		\end{equation*}
		The second term in the right-hand side above vanishes since the integrand is odd. Now, a straightforward computation gives
		\begin{equation*}
			\begin{aligned}
				\mathcal{F}^{-1}(H_s)(\xi) = & \int_{\R} \cos(2\pi \xi t) \frac{\sinh(st)}{\sinh t}\,d t = \frac{\pi\sin(s\pi)}{\cos(s\pi) + \cosh(2\pi^{2}\xi)} \geq 0,
			\end{aligned}
		\end{equation*}
		which shows our claim.
		
		%%%%%%%%%%%%%%%%%%
		\noindent\textsc{Step 6.} \textit{The case of repeated eigenvalues.} We will indicate the necessary changes in the proof above.
		
		If $\lambda_{j}=\lambda_{k}$ for some $j\neq k$, then the right-hand side of \eqref{consequence classical Pohozaev with lambdas} vanishes, and we can only conclude antisymmetry from \eqref{antisymmetry.1}, that is,
		\begin{equation}\label{antisymmetry}
			\int_{\Omega}\lp x \cdot \nabla \varphi_{j}\rp \varphi_{k} \, dx =- \int_{\Omega}\lp x \cdot \nabla \varphi_{k}\rp \varphi_{j}\, dx.
		\end{equation}
		In particular, it follows from \eqref{plugging eigenfunctions}
		that
		\begin{equation*}
			\hQ 1 j k=0.
		\end{equation*}
		In this case, the argument in Step 3 is modified as follows. From \eqref{Q tilde} we obtain
		$$\tilde{Q}^{(s)}_{j k}=
		- \lambda_{k}^{s}\int_{\Omega} \lp x \cdot \nabla \varphi_{j}\rp \varphi_{k}\,dx,$$
		and then by \eqref{antisymmetry}
		\begin{equation*}
			\hQ s j k = \frac{1}{2}\lp- \lambda_{k}^{s}\int_{\Omega} \lp x \cdot \nabla \varphi_{j}\rp \varphi_{k}\,dx + \lambda_{j}^{s}\int_{\Omega} \lp x \cdot \nabla \varphi_{j}\rp \varphi_{k}\,dx\rp = 0.
		\end{equation*}
		Then, the Schur product structure \eqref{conclusion-structure} holds trivially since both $\hQ 1 jk$ and $\hQ s jk$ vanish. This completes the proof for the case of repeated eigenvalues.
		
	\end{proof}
	
	%%%%%%%%%%%%%%%%%%%%%%%%
	\section{Proofs of non-existence results}
	
	The strategy to prove non-existence results from the Pohozaev identity is analogous to the classical case. First, we deduce an inequality for solutions to semilinear equations.
	
	\begin{proof}[Proof of Proposition \ref{prop:Pohozaev-semilinear}] This is a straightforward corollary of Theorem \ref{Pohozaev SFL}.
		
		\noindent First, note that since $u$ is bounded, $|u| \leq M$, $f \in C^{0,1}(\overline{\Omega}\times[-M,M])$. Then, by \cite{LionsNonhomoMagenes1972,AdamsSobolevFournier2003} (see also \cite[Appendix]{bonforte2014existence}), and the fact that if $u \in H^{\alpha}$, $f(x,u) \in H^{\min\{1,\alpha\}}$, we can bootstrap the regularity as follows:
		\begin{equation*}
			u \in L^{2} \ \Rightarrow \ f(x,u) \in L^{2} \ \Rightarrow \ u \in H^{2s} \ \Rightarrow \ f(x,u) \in H^{\min\{1,2s\}} \ \Rightarrow \ \ldots \ \Rightarrow \ u \in H^{1+2s}.
		\end{equation*}
		
		When $s \geq 1/4$, we have $1+2s \geq 3/2$, so $u \in H^{1+2s} \subseteq H^{\max[2s, 3/2]}$ and Theorem \ref{Pohozaev SFL} applies directly. When $s < 1/4$, we have $u \in H^{1+2s}$ but $1+2s < 3/2$. In this case, we use an approximation argument. Let $u_n \in H^{3/2}_0(\Omega)$ be a sequence approximating $u$ in $H^{1+2s}$. For each $n$, Theorem \ref{Pohozaev SFL} gives $Q^{(s)}[u_n] \geq 0$. Since $u \in H^{1+2s}$ ensures that $Q^{(s)}[u]$ is well-defined and finite, and since $Q^{(s)}$ is a continuous bilinear form on $H^{1+2s} \times H^{1+2s}$, we have $Q^{(s)}[u_n] \to Q^{(s)}[u]$ as $n \to \infty$. Taking the limit yields $Q^{(s)}[u] \geq 0$.
		
		Once we have the positivity of $\Q s [u]$, the rest of the argument is classical. Indeed, by the definition of $F$ we know that,
		\begin{equation*}
			\nabla F(x,u) = F_{x}(x,u) + \nabla u \cdot F_{u}(x,u) = F_{x}(x,u) + \nabla u \cdot f(x,u),
		\end{equation*}
		thus since $u$ is a solution to problem \eqref{pb SFL dependence x},
		\begin{equation*}
			(x\cdot \nabla u) (-\Delta|_{\Omega})^{s}u = (x\cdot \nabla u) f(x,u) = x\cdot \nabla F(x,u) - x\cdot F_{x}(x,u).
		\end{equation*}
		Knowing that $ N = \operatorname{div}(x)$ and integrating by parts,
		\begin{equation*}
			\int_{\Omega}(x\cdot \nabla u) (-\Delta|_{\Omega})^{s}u\,dx = \int_{\Omega}(x\cdot \nabla u) f(x,u)\,dx = -N\int_{\Omega}F(x, u)\,dx -\int_{\Omega} x\cdot F_{x}(x, u) \, dx.
		\end{equation*}
		The conclusion follows from the fact that $\Q s [u]\geq 0$.
	\end{proof}
	
	Our non-existence results easily follow from contradicting the sign condition that comes from the Pohozaev identity. We prove the first of them for the sake of illustration.
	\begin{proof}[Proof of Corollary \ref{cor: non existence SFL dependence x}]	
		Let $u$ be a bounded non-trivial solution to problem \eqref{pb SFL dependence x}. Then, by Proposition \ref{prop:Pohozaev-semilinear},
		\begin{equation*}
			\lp\frac{2s-N}{2}\rp \int_{\Omega} uf(x,u)\,dx + N\int_{\Omega}F(x, u)\,dx + \int_{\Omega} x\cdot F_{x}(x, u) \, dx \geq 0.
		\end{equation*}
		Moreover, we are assuming $u \not\equiv 0$ and
		$$\lp\frac{2s-N}{2}\rp tf(x, t) + NF(x,t) + x\cdot F_{x}(x,t) < 0, \quad \mbox{for all }x\in\Omega, \mbox{and } t\in \R\setminus\{0\},$$
		and hence
		$$\lp\frac{2s-N}{2}\rp \int_{\Omega} uf(x,u)\,dx + N\int_{\Omega}F(x, u)\,dx + \int_{\Omega} x\cdot F_{x}(x, u) \, dx < 0,$$
		a contradiction.
	\end{proof}
	
	Finally, Corollary \ref{cor: non-existence SFL} follows from Corollary \ref{cor: non existence SFL dependence x} by taking $f$ independent of $x$, and Corollary \ref{cor: non-existence semilinear pb} follows by specializing to the power nonlinearity $f(t) = |t|^{p-1}t$.
	
	%%%%%%%%%%%%%%%
	\subsection{Acknowledgments.} I.B-C., M.B. and M.G. acknowledge financial support from the grants PID2020-113596GBI00, PID2023-150166NB-I00, RED2022-134784-T, all funded by Ministry of Science and Innovation, Spain, MCIN/ AEI/10.13039/501100011033, the “Severo Ochoa Programme for Centers of Excellence in R\&D” (CEX2019-000904-S), and the financial support from the FPI-grant PRE2021-097207 (Spain).
	
	\noindent C.T.-L. has received funding from the European Research Council (ERC) under the Grant Agreement No 862342, from AEI project PID2024-156429NB-I00 (Spain), and from the Grant CEX2023-001347-S funded by MICIU/AEI/10.13039/501100011033 (Spain).
	
	%%%%%%%%%%%%%%%%%%%%%%%
	
	%%%% BIBLIOGRAPHY %%%%
	%\nocite{*}

	%\newpage
	\authorinfo
	{Itahisa Barrios-Cubas}
	{Departamento de Matem\'{a}ticas, Universidad Aut\'{o}noma de Madrid,
		\newline ICMAT - Instituto de Ciencias Matem\'{a}ticas, CSIC-UAM-UC3M-UCM,
		\newline Campus de Cantoblanco, 28049 Madrid, Spain.}
	{itahisa.barrios@uam.es}
	
	\authorinfo
	{Matteo Bonforte}
	{Departamento de Matem\'{a}ticas, Universidad Aut\'{o}noma de Madrid,
		\newline ICMAT - Instituto de Ciencias Matem\'{a}ticas, CSIC-UAM-UC3M-UCM,
		\newline Campus de Cantoblanco, 28049 Madrid, Spain.}
	{matteo.bonforte@uam.es}
	
	\authorinfo
	{Mar\'{i}a del Mar Gonz\'{a}lez}
	{Departamento de Matem\'{a}ticas, Universidad Aut\'{o}noma de Madrid,
		\newline ICMAT - Instituto de Ciencias Matem\'{a}ticas, CSIC-UAM-UC3M-UCM,
		\newline Campus de Cantoblanco, 28049 Madrid, Spain.}
	{mariamar.gonzalezn@uam.es}
	
	\authorinfo
	{Clara Torres-Latorre}
	{Instituto de Ciencias Matem\'{a}ticas,
		\newline Consejo Superior de Investigaciones Cient\'{i}ficas,
		\newline C/ Nicol\'{a}s Cabrera, 13-15, 28049 Madrid, Spain.}
	{clara.torres@icmat.es}
	
\end{document}